\documentclass[letterpaper, 11pt]{amsart}
\usepackage{amsmath,amssymb,amsxtra,amsthm, amstext, amscd,amsfonts,fancyhdr, hyperref,enumerate, longtable, mathrsfs}
\usepackage[all,cmtip]{xy}
\hypersetup{
colorlinks=true,
urlcolor=black,
citecolor=blue,
linkcolor=blue,
}

\usepackage[OT2,T1]{fontenc}
\DeclareSymbolFont{cyrletters}{OT2}{wncyr}{m}{n}
\DeclareMathSymbol{\Sha}{\mathalpha}{cyrletters}{"58}

\newcommand{\bC}{{\mathbb{C}}}

\newcommand{\bF}{{\mathbb{F}}}

\newcommand{\bP}{{\mathbb{P}}}
\newcommand{\bQ}{{\mathbb{Q}}}

\newcommand{\bZ}{{\mathbb{Z}}}

\newcommand{\Bz}{{\mathbf{z}}}

\newcommand{\BB}{\mathbf{B}}

  \newcommand{\C}{{\mathcal{C}}}

  \newcommand{\F}{{\mathcal{F}}}
  \newcommand{\G}{{\mathcal{G}}}

  \newcommand{\K}{{\mathcal{K}}}  
\renewcommand{\L}{{\mathcal{L}}}

\renewcommand{\O}{{\mathcal{O}}}

\renewcommand{\S}{{\mathcal{S}}}
  
  \newcommand{\UUU}{{\mathcal{U}}}

  \newcommand{\X}{{\mathcal{X}}}
  \newcommand{\Y}{{\mathcal{Y}}}
  \newcommand{\Z}{{\mathcal{Z}}}
  \newcommand{\HH}{\mathcal{H}}

\newcommand{\fp}{\mathfrak{p}}

\newcommand{\fm}{\mathfrak{m}}

\newcommand{\fY}{\mathfrak{Y}}

\newcommand{\fH}{\mathfrak{H}}

\newcommand{\Hom}{\operatorname{Hom}}

\newcommand{\ord}{\operatorname{ord}}
\newcommand{\Gal}{\operatorname{Gal}}

\newcommand{\GL}{\operatorname{GL}}

\newcommand{\Aut}{\operatorname{Aut}}

\newcommand{\HHH}{\mathcal{H}}

\newcommand{\fS}{\mathfrak{S}}
\newcommand{\fR}{\mathfrak{R}}

\newcommand{\fF}{\mathfrak{F}}
\newcommand{\ep}{\varepsilon}
\newcommand{\Jac}{\operatorname{Jac}}
\newcommand{\dR}{\operatorname{dR}}

\newcommand{\Gr}{\operatorname{Gr}}

\newcommand{\Aff}{\operatorname{Aff}}
\newcommand{\codim}{\operatorname{codim}}
\newcommand{\GM}{\operatorname{GM}}
 
\newcommand{\size}{\operatorname{size}}
\newcommand{\Frob}{\operatorname{Frob}}

\newcommand{\Spec}{\operatorname{Spec}}

\newcommand{\ol}{\overline}

\newcommand{\upchi}{{\raise.35ex\hbox{$\chi$}}}

\newcommand{\SL}{\operatorname{SL}}

\newcommand{\prim}{\operatorname{prim}}
\newcommand{\ra}{\rightarrow}

\newcommand{\crys}{\textnormal{crys}}

\makeatletter
\@namedef{subjclassname@2010}{%
  \textup{2010} Mathematics Subject Classification}
\makeatother

\newtheorem{theorem}{Theorem}[section]
\newtheorem{corollary}[theorem]{Corollary}
\newtheorem{proposition}[theorem]{Proposition}
\newtheorem{lemma}[theorem]{Lemma}
\newtheorem{conjecture}[theorem]{Conjecture}
\newtheorem{claim}[theorem]{Claim}

\theoremstyle{definition}

\theoremstyle{remark}
\newtheorem{remark}[theorem]{Remark}

\numberwithin{equation}{section}

\allowdisplaybreaks

\begin{document}

\title{Uniformity of fibres of period mappings and the $S$-unit equation}

%On quartic $D_4$-field with monogenic cubic resolvent

%\author{Lucia Mocz}
%\address{Hausdorff Center for Mathematics \\
%Villa Maria \\
%Endenicher Allee 62 \\
%D-53115 Bonn, Germany}
%\email{lmocz@math.uni-bonn.de}

\author{Brett Nasserden}
\address{Department of Pure Mathematics \\
University of Waterloo \\
200 University Ave W \\ 
Waterloo, Ontario, Canada \\ N2L 3G1}
\email{bnasserd@uwaterloo.ca}

\author{Stanley Yao Xiao}
\address{Department of Mathematics \\
University of Toronto \\
Bahen Centre \\
40 St. George Street, Room 6290 \\
Toronto, Ontario, Canada \\  M5S 2E4 }
\email{syxiao@math.toronto.edu}
\indent

%%%%%%%%%%%%%%%%%%%%%%%%%%%%%%%%%%%%%%%%%%%%%%%%%%%%%%%%%%%%%%%%%%

\begin{abstract} In this paper we give a refinement of the method introduced by Lawrence and Venkatesh in \cite{LV} and thereby showing that their proof of Mordell's conjecture is uniform up to a uniform bound on the number of Galois representations attached to some family of abelian varieties. We are also able to give an unconditional proof of a uniform boundedness statement on the number of solutions to $S$-unit equations, which qualitatively is best possible, recovering a result of Evertse \cite{Eve}. 
\end{abstract}

\maketitle

\vspace{-5mm}

%%%%%%%%%%%%

\section{Introduction}

One of the most stunning and celebrated achievements in arithmetic geometry, and number theory in general, is Faltings' famous proof of Mordell's conjecture asserting that every algebraic curve $C$ defined over a number field $K$ having genus $g$ exceeding one has at most finitely many $K$-rational points. Faltings' proof is even more impressive when one takes into account the fact that he also proved Shafarevich's conjecture and Tate's conjecture for abelian varieties in the same paper \cite{Fal}. \\

In \cite{CHM} Caporaso, Mazur, and Harris proved a surprising result: assuming the purely qualitative Bombieri-Lang conjecture, asserting that for any variety $V$ of general type defined over a number field $K$, the set of $K$-rational points on $V$ is not Zariski dense, they proved that Faltings' theorem is \emph{uniform}, in the sense that for any number field $K$ and integer $g \geq 2$ there exists a positive $N(K,g)$ such that for any algebraic curve $C$ defined over $K$ having genus $g$, the cardinality of $C(K)$ is bounded by $N(K,g)$. This assertion is known as the \emph{uniform boundedness conjecture}. \\

In \cite{Sto} and \cite{KRZB} Stoll and Katz, Rabinoff, and Zuerick-Brown proved respectively that the uniform boundedness conjecture holds for hyperelliptic curves and arbitrary algebraic curves respectively, provided that one assumes in addition that the rank $r = r(K)$ of the Jacobian variety of $C$ satisfies $r \leq g - 3$. By the work of Bhargava and Gross \cite{BG}, this hypothesis is expected to hold for a proportion tending to 100\% of hyperelliptic curves defined over $\bQ$ as the genus $g$ tends to infinity. \\

The papers \cite{Sto} and \cite{KRZB} both employ the so-called \emph{Chabauty-Coleman} method. The fundamental restriction is that one must assume that $r(K) < g$ in order for the argument to run. M.~Kim has proposed, starting in \cite{Kim1}, a remarkable generalization of Chabauty's method which replaces the Mordell-Weil lattice $\Jac(Y)(K)$ and the $v$-adic Jacobian $\Jac(Y)(K_v)$ by non-abelian quotients of the unipotent fundamental group cut out by certain local conditions, which he calls the global and local \emph{Selmer varieties} respectively. If the dimension of the global Selmer variety is less than the corresponding local Selmer variety, then one can run Chabauty's argument to prove finiteness. Several authors, including J.~Balakrishnan and N.~Dogra have immensely pushed Kim's program forward, culminating in some truly stunning results, such as \cite{BDMTV}. \\

In 2018, B.~Lawrence and A.~Venkatesh introduced a remarkable new approach to Mordell's conjecture in \cite{LV} which is distinct from both the Chabauty-Coleman-Kim approach and Faltings' proof \cite{Fal}. Morally speaking, Lawrence and Venkatesh's approach replaces the dependency on the rank of the Mordell-Weil lattice or more generally the dimension of a global Selmer variety by the finiteness of global Galois representations of pure weight having good reduction outside of a set of primes $S$. The latter is similar to Faltings' original approach, but the added benefit is that Lawrence and Venkatesh's set-up is more flexible and conducive to variation in some families. \\

The goal of this paper is to show that, indeed, the methods introduced by Lawrence and Venkatesh offer at least some flexibility. In principle, this reduces the question of uniform boundedness of rational points to the uniform finiteness of some families of Galois representations. We propose the following conjecture in that regard: 

\begin{conjecture} \label{conj} Let $K$ be a number field, $S$ a set of finite places of $K$, and $g \geq 2$ an integer. Let $\F$ be a one-parameter family of complex abelian varieties $A$ of dimension $g$, parametrized by an algebraic curve $\C$ defined over $K$. Let $\UUU(K, S, \F)$ be the set of ($K$-isogeny classes of) abelian varieties $[A]_K$, defined over $K$ and having good reduction outside of $S$, where there is some $A^\prime\in [A]_K$ with $A^\prime(\bC) \in \F$. Then $|\UUU(K, S, \F)|$ depends at most on $K$ and $g$. 
\end{conjecture}

While Conjecture \ref{conj} appears somewhat unmotivated, it is actually exactly what is needed to prove the following uniformity result:

\begin{theorem} \label{MT} Let $K$ be a number field, $v$ a finite place of $K$, and $g \geq 2$ an integer. Let $\fY$ be a complex algebraic curve of genus $g$, and let $\Y_v(K)$ denote the set of curves $Y/K$ which are isomorphic to $\fY$ over $\bC$ and have good reduction at $v$. \\

Assuming Conjecture \ref{conj}, there exists a number $N(\fY, K,g,v)$ which depends only on $\fY, K, g, v$ such that for any $Y \in \Y(K)$, the number of $K$-rational points on $Y$ is bounded by $N(\fY, K, g, v)$.  
\end{theorem}

Theorem \ref{MT} is derived from our technical refinement of the arguments in \cite{LV}. In \cite{LV}, Lawrence and Venkatesh defines a \emph{$p$-adic period map} from the space of semi-simple representations $\Gal(\ol{\bQ_p}/K_v)$ to a certain flag variety $\fH$. Indeed, we obtain such a result as our Proposition \ref{not zar}. \\

Theorem \ref{MT} implies that we have uniform boundedness in any \emph{twist family}. Two immediate applications come to mind: 

\begin{corollary} \label{cor1} Let $F$ be a binary form of degree $2g + 2, g \geq 2$ having integer coefficients and irreducible over $\bQ$. If Conjecture \ref{conj} is true, then there exists a positive integer $N(F)$ such that for any non-zero square-free integer $d$, the number of rational points on the curve
\[dy^2 = F(u,v)\]
is at most $N(F)$.
\end{corollary}

\begin{proof} The form $F$ uniquely determines the complex isomorphism class and hence the genus, and in this case the field of definition is $K = \bQ$. Thus it suffices to explain why there is no dependence on an auxiliary prime $v$. To see this, choose the smallest odd prime $p \nmid \Delta(F)$ with the property that $F(x,y) \equiv 0 \pmod{p}$ has no solution. Note that $y^2 = F(u,v)$ has good reduction at $p$. Applying the Theorem with $v = p$ covers all $d$ with $p \nmid d$. For those $d$ divisible by $p$, our choice of $p$ implies that the twist cannot have rational points due to local obstruction. Since $p$ is uniquely determined by $F$, the dependency can be removed. 
\end{proof}

Corollary \ref{cor1} implies a weak version of a conjecture of C.~L.~Stewart in \cite{Stew}. In particular, the number of primitive solutions to the Thue equation
\begin{equation} \label{Thue} F(x,y) = h
\end{equation}
depends at most on $F$, provided that $\deg F \geq 6$ and $F$ is irreducible over $\bQ$. Currently, uniform bounds were only known when $h$ is fixed, due to Bombieri and Schmidt in \cite{BS} and later extended by Stewart \cite{Stew}. Stewart and the second author showed that Stewart's conjecture holds true on average in \cite{SX}. \\

We now give a heuristic as to why Conjecture \ref{conj} is plausible. Consider elliptic curves $E_{A,B}/\bQ$, given by the minimal Weierstrass model:
\begin{equation} \label{Weir} E_{A,B} : y^2 = x^3 + Ax + B, A,B \in \bZ, p^4 | A \Rightarrow p^6 \nmid B \text{ for all primes } p.
\end{equation}
We can then view $\bZ^2$ as a parameter space for the $E_{A,B}$'s. For a set $S$ of rational primes, a curve $E_{A,B}$ has good reduction outside of $S$ if the discriminant $\Delta(E_{A,B})$ is only divisible by primes in $S$. There are only a finite number of integers which are composed only of primes in $S$ which can be discriminants of elliptic curves in minimal Weierstrass model, say $D_1(S), \cdots, D_m(S)$. Thus, an elliptic curve $E_{A,B}$ has good reduction outside of $S$ only if $(A,B)$ is an integer point one of the Mordell curves
\begin{equation} \label{mor cur} 27y^2 = 4x^3 + D_j(S), j = 1, \cdots, m.
\end{equation}
Let $\C$ be a planar curve of genus $g \geq 2$, which corresponds to a 1-parameter family of complex elliptic curves. Then the number of integral intersection points of these curves with $\C$ is at most the number of integral points of $\C$, given as an affine plane curve.\\

One case where we may say something unconditionally is the case of $S$-unit equations. This is a well-studied problem, going back to the work of Siegel \cite{Sieg}. Kim gave a new proof of this finiteness result as an introduction to his non-abelian Chabauty method in \cite{Kim1}, and Lawrence and Venkatesh gave yet another proof in \cite{LV} as a proof-of-concept of their method. \\

Following \cite{LV} and our improvements made to their set-up, we prove the following uniform variant:

\begin{theorem} \label{SMT} Let $K$ be a number field which contains a primitive $8$-th root of unity. Let $S$ be a finite set of finite places of $K$ which contains all places above $2$. Let $m$ be the largest power of $2$ which divides the order of the group of roots of unity in $K$. Let $L$ be a fixed cyclic extension of $K$ of degree $m$. Choose a place $v \not \in S$ which is inert in $L$. Then there exists a number $N(K, L, v)$ such that for any $t_0 \pmod{v}$, the set
\[U_{1,L,v, t_0} = \{t \in \O_{S}^\ast : 1 - t \in \O_S^\ast, t \not\in (K^\ast)^2, K(t^{1/m}) \cong L, t \equiv t_0 \pmod{v}\}\]
we have $|U_{1,L,v,t_0}| \leq N(K,L,v)$. 
\end{theorem} 

In other words, the number of solutions to the $S$-unit equation is uniformly bounded, up to the number of possibilities of cyclic extensions $L$ that can arise, which depends only on the cardinality of $S$. Such a dependency is unavoidable; see \cite{EST}. This theorem is quantitatively weaker than a theorem of Evertse \cite{Eve}, but is qualitatively similar. \\

In fact, we can give a concrete bound on the number of fields $L = K(t^{1/m})$ that can arise. Note that for any such $L$, the (fractional) principal ideal $\left(t^{1/m}\right)$ factors into prime ideals that are totally ramified over $K$, and hence the prime divisors of $(t^{1/m})$ are in one-to-one correspondence to the prime ideals in $\O_K$ which occur in the factorization of $(t)$ into a fractional ideal. There are now two situations: for those primes co-prime with $m$ only tame ramification is possible, and for each such prime $\fp \in S$, the possible exponents occurring in $(t)$ are (up to equivalence) $0, \cdots, m-1$. \\

In our argument we may fix $m = 8$, which we do so now. If we write $(t) = \fp_1^{\ell_1} \cdots \fp_{|S|}^{\ell_{|S|}}$, then for each $1 \leq k \leq |S|$ there are $8$ possibilities for $\ell_k$, and thus there are $8^{|S|}$ possibilities for $(t)$, assuming there is no wild ramification. To account for possible wild ramification we need to multiply by an absolute constant $C_0$. This gives that there is a positive number $C_0$ such that the number of possible fields $L$ that may occur is at most $C_0 \cdot 8^{|S|}$. \\

The dependence on the specific field $L$ in Theorem \ref{SMT} is quite mild, in the sense that we only require $v$ to be inert in $L$. Fixing such a prime and choosing only fields $L = K(t^{1/m})$ for which $v$ is inert in $L$, we obtain a bound of the shape
\begin{equation} \# \{t \in \O_{K,S}^\ast : 1 - t \in \O_{K,S}^\ast, v \text{ inert in } K(t^{1/m}) \} \leq N(K,v) 8^{|S|} \end{equation}
for some number $N(K,v)$ which only depends on $K$ and $v$. \\

The main thrust behind Theorem \ref{SMT} is that we can replace Conjecture \ref{conj} with a statement about 2-dimensional representations of $\Gal(\ol{\bQ_p}/K_v)$. Indeed, 2-dimensional crystalline representations of $\Gal(\ol{\bQ_p}/K_v)$ with Hodge-Tate weight $\{0,1\}$ are completely determined by their characteristic polynomial.

\subsection*{Acknowledgements} We thank Lucia Mocz for discussions that formed the genesis of this project. We thank Minhyong Kim for meeting with us during his visit to the Perimeter Institute and discussing relevant ideas, which ultimately gave rise to the current project. We thank him again for reading earlier versions of this manuscript which improved the quality immeasurably.

\section{The general set up}
\label{gen set}

Here we begin discussing the general framework that was described in the \cite{LV} paper. We work in somewhat larger generality than what is needed in \cite{LV}. As we proceed further we will add extra assumptions as needed. Hopefully, this will clarify the ideas we are using. \\

Suppose that we have a smooth morphism

\[\psi \colon X\ra Y\]
of varieties over a number field $K$ of relative dimension $d$. We are primarily interested in the case when $\Psi\colon X\ra Y$ is given the structure of a semi-abelian scheme. This has the benefit of making the types of Galois representations arising as the typical ones coming from the representation of $\Gal(\ol{\bQ}/K)$ on Tate modules of abelian varieties, for some prime $p$.  Suppose further that we have a set $S$ of primes such that $\psi$ extends to a smooth morphism of $\O_S$ schemes, say

\[\Psi\colon \X\ra \Y.\]

Choose a prime $v$ outside of $S$, lying above $p \in \bQ$, and such that \begin{enumerate}
	\item $p>2$;
	\item $p$ does not lie below any prime in $S$; and
	\item $Y$ has good reduction at $v$. 
	\end{enumerate} 
For each $y\in \Y(\O)$ we have a a Galois representation $\rho_y$ associated to the action of the absolute Galois group $G_K$ on $H^q(X_y\times_K\bar{K},\bQ_p )$. In fact more can be said: if $y\in \mathcal{Y}(\O)$ then $\rho_y$ is a \emph{crystalline} representation when restricted to $K_v$. This follows from the assumption that we have a good model. With these restrictions we obtain a map

\begin{equation} \label{rep map} r_{\Psi,v,q}=r\colon Y(\O)\ra \{\textnormal{crystalline representations of }G_{K_v}\textnormal{ on }\bQ_p\textnormal{ vector spaces}\}\end{equation}
that sends 
\[y\mapsto \rho_y\mid_{G_{K_v}}.\]
On the other hand, $p$-adic Hodge theory gives a map

\[\textnormal{crys}\colon \{\textnormal{crystalline representations of }G_{K_v}\textnormal{ on }\bQ_p\textnormal{ vector spaces}\}\ra \F\L.\] 
Here $\F\L$ is the category of triples $(W,\phi,F)$ , where $W$ is a $K_v$ vector space, $\phi$ a Frobenius semi-linear automorphism of $W$, and $F$ a descending filtration on $W$. The crystalline comparison theorem then shows that for $y\in Y(\O)$ we have

\[\crys\circ r(y)=(H^q(X_y/K_v,\textnormal{Fr}_v),\textnormal{Hodge Filtration for }X_y)\]
That is, the image of $\rho_y$ inside $\F\L$ is $(H^q(X_y/K_v,\textnormal{Fr}_v),\textnormal{Hodge Filtration for }X_y)$. The comparison theorem with lifts of crystalline cohomology gives that the triple \[(H^q(X_y/K_v,\textnormal{Fr}_v),\textnormal{Hodge Filtration for }X_y)\cong (V_v,\phi_v,\Phi_v(y)) \] where $V_v$ is a lift of the crystalline cohomology, and $\phi_v$ is the canonical Frobenius operator. From the mapping $\Psi$ we have essentially constructed a category of linear algebraic data, which we may hope to study in lieu of studying $Y$ itself. Since we are now working with a category of linear algebraic data over a local field, we study how points that are $v$-adically close behaves. To this end we consider those points of $\F\L$ that are close to point of $\mathcal{R}$. To do this in a controlled way we consider residues disks on $Y(\O)$ and \emph{period mappings} induced by the crystalline comparison theorem. \\

Fix $y_0\in Y(\O)$ and put $\Omega_v(y_0)=\{y\in Y(\O_v)\colon y\equiv y_0\pmod{ v}\}$, the residue disk of $Y$ at $y_0$. If $y\in Y(\O_v)$ and $y\equiv y_0 \pmod{v}$ then the Gauss-Manin connection gives an isomorohism of vector spaces

\begin{equation} \label{GM} \GM \colon H^q(X_y/K_v)\cong H^q(X_{y_0}/K_v). \end{equation}
We remark that when $Y$ is defined over a number field $K$, that the Gauss-Manin connection (\ref{GM}) is given by a power series with $K$-coefficients. \\

The map (\ref{GM}) preserves the Frobenius mapping $\phi_v$ but varies the Hodge filtration, since the latter is dependent on the particular lift to $K_{v}$ and the former depends only on the residue class modulo $v$. We assume that all the Hodge filtrations constructed from $\Psi$ have the same dimensional data, since all such filtrations arising from global representations necessarily satisfy this property. We then obtain a morphism 

\[\phi_{y_0}\colon \Omega_{v}(y_0)\ra \HH(K_v)\] that takes the Hodge filtration associated to $X_y$ to its image under the Gauss-Manin connection. We obtain a map

\[\Phi_{\pi,v,q,y_0}=\Phi\colon \Omega_v(y_0)\ra \F\L,y\mapsto (V_v,\phi_v,\phi_{y_0}(y))\]

The projection unto the first factor gives recovers the so called \emph{period mapping}
\[\phi_{y_0}\colon \Omega_v(y_0)\ra\HHH(K_v),y\mapsto \phi_{y_0}(y)\]
where $\HHH(K_v)$ is the flag variety parametrizing the subspaces with the correct data. Set \[\mathfrak{R}_{\pi,v,q,y_0}=\mathfrak{R}=\Phi(\Omega_v(y_0))\]

Let $\S$ be the category of $K_v$ vector spaces with a Frobenius semi-linear operator. That is, the category whose objects are pairs $(V,\phi)$ where $V$ is a finite dimensional $K_v$ vector space and $\phi$ is a Frobenius semi-linear operator on $V$. The morphisms in $\S$ are given by linear maps that respect the semi-linear operators. Notice that $\F\L$ comes with a natural forgetful map

\[p_{12}\colon \F\L\ra \S,(V,\phi,F)\mapsto (V,\phi)\]

\begin{proposition}
Let $\mathfrak{R}$ be as above. Then $\mathfrak{R}$ lies in a fiber of $p_{12}$.
\end{proposition}
\begin{proof}
By construction, all the elements of $\mathfrak{R}$ are of the form
\[(V_v,\phi_v,\phi_{y_0}(y))\]
for some $y\in \Omega_v(y_0)$. Thus $\mathfrak{R}\subseteq p_{12}^{-1}(V_v,\phi_v)$.
\end{proof}

Thus the natural category of linear algebraic data to work with is the full fiber over $(V_v,\phi_v)$ which we will denote $\mathfrak{F}$. Given a point $x\in \mathfrak{F}$ the collection of $x^\prime\in \fF$ isomorphic to $x$ is given by the collection of all linear automorphisms of $V_v$ which commute with $\phi_v$. We call this the centralizer of $\phi_v$ and denote it $Z(\phi_v)$. Let $p_3\colon \fF\ra \HH(K_v)$, and take to $[x]$ be the collection of elements in $\fF$ which are isomorphic to $x$. We have  \[[x]=\{(V_v,\phi_v,T(x):T\in Z(\phi_v)\}=p_3^{-1}(Z(\phi_v)\cdot p_3(x) )\]
Even more explicitly, we have that the orbit of $p_3(x)$ in $\HH(K_v)$ can be described as \[Z(\phi)\cdot p_3(x)\cong Z(\phi_v)/\textnormal{Stab}_{Z(\phi_v)}(p_3(x)):=\G_x\]
Then 
\[[x]=p_3^{-1}(G_x)\]  so that the isomorphism classes of a point in $\fF$ arise as fibres over certain linear algebraic groups in a space of flags a general point $x\in \fF$. In fact, the period morphism $\phi_{y_0}$ is a $K_v$ analytic mapping and our most interesting results will be about the fibres 
\[\fF_y=\phi_{y_0}^{-1}(\G_{\Phi(y)}).\]
In particular when $Y$ is a curve we will show that $\fF_y$ is finite, and give upper bounds on the size of these fibres with limited dependencies, as to give us freedom to vary over a family uniformly. So far we have only been using the fact that we have a non-archimedean place of good reduction. Now pick an archimedean place $\iota$ and consider the monodromy representation

\[\mu_{\iota,q}=\mu\colon \pi_1(Y(\bC),y_0)\ra \textnormal{GL}(H^q(X_{y_0}\times_K \bar{K},\bar{\bQ})\times_K\bC)\]

We take \[\Gamma_\mu=\Gamma=\textnormal{Zariski closure of }\textnormal{image}(\mu)\]

Now let $h_0$ be the Hodge filtration coming from the complex de Rham cohomology $H^q(X_{y_0}(\bC)/\bC)$. Thus $h_0\in \HH(\bC)$ where $\HH$ is the flag variety parametrizing the appropriate dimensional data. We set

\[D_\mu=\dim_\bC \Gamma_\mu\cdot h_0\]

We may now state our first result.

\begin{lemma}
With the notation above we have that \[\dim\overline{\phi_{y_0}(\Omega_v(y_0))}\geq D_\mu.\]
In words, the Zariski closure of the image of $\phi_{y_0}$ has dimension at least the dimension of $\Gamma_\mu\cdot h_0$. In particular, if $W\subseteq \HH_v$ is closed with $\dim W<D_\mu$  then $\phi_{y_0}^{-1}(W)$ is contained in a proper Zariski closed subset of $\Omega_v(y_0)$. 

\end{lemma}
\begin{proof}
This is a consequence of Lemma 3.1 in \cite{LV}.
\end{proof}

The main goal of this section is to prove the following refinement of Proposition 3.3 in \cite{LV}:

\begin{proposition} \label{not zar} Let $X \rightarrow Y$ be a smooth, proper family over a number field $K$, $V_q$ the degree $q$ de Rham cohomology of a given fibre $X_0$ above $y_0 \in Y(K)$, $\HH$ a space of flags in $V_q$, 
\[\Phi_v : \{y \in \Y(\O_v) : y \equiv y_0\} \rightarrow \HH(K_v)\]
the $v$-adic period mapping as defined above, $\Gamma \subset \GL(V_q(\bC))$ the Zariski closure of the monodromy group, and $h_0 = \Phi(y_0)$ is the image of $y_0$ under the period mapping. If 
\[\dim_{K_v} \left(Z\left(\phi_v^{[K_v : \bQ_p]} \right) \right) < \dim_{\bC} \Gamma \cdot h_0^\iota,\]
where $Z(\cdot)$ denotes the centralizer, in $\Aut_{K_v}(V_q(K_v))$, of the $K_v$-linear operator $\phi_v^{[K_v : \bQ_p]}$, then for a fixed crystalline representation $\rho$ of $G_{K_v}$ the set
\[\{y \in Y(\O) : y \equiv y_0 \pmod{v}, \rho_y \cong \rho\}\]
is contained in a proper algebraic subvariety of the residue disk $Y(K_v)$ at $y_0$. Moreover, the cardinality depends only on the field $K$, the prime $v$, and the $K_v$-isomorphism class of $X \rightarrow Y$.
\end{proposition}

The proof of Proposition \ref{not zar} is a consequence of the following lemma, which is essentially Theorem 9.1 and Corollary 9.2 in \cite{LV}: 

\begin{lemma} \label{BT lem} Suppose $V \subset Y \times \fH^\ast$ is an algebraic set. Write $W$ for the image of $\tilde{Y}$, the universal cover of $Y(\bC)$, in $Y \times \fH$. Suppose that $U \subset V \times W$ is an irreducible analytic set such that
\[\codim_{Y \times \fH^\ast} U < \codim_{Y \times \fH^\ast} V + \codim_{Y \times \fH^\ast} W,\]
where all codimensions are tkane inside $Y \times \fH^\ast$. Then the projection of $U$ to $Y$ is contained in a proper weak Mumford-Tate subvariety. \\

If $Z \subset \fH^\ast$ is an algebraic subvariety and 
\begin{equation} \label{codim Z} \codim_{\fH^\ast} (Z) \geq \dim (Y),
\end{equation}
then any irreducible component of $\Phi^{-1}(Z)$ is contained inside the preimage, in $\tilde{Y}$, of the complex points of a proper subvariety of $Y$.
\end{lemma}

\begin{proof} Let $Q$ be an irreducible component of $\Phi^{-1}(Z)$. Let $V = Y \times Z$. The intersection $W^Z$ of $W$ with $Y \times (Z \cap \fH)$, taken inside $Y \times \fH$, is an analytic set. Moreover, the image of $Q$ under the analytic map $\tilde{Y} \rightarrow Y \times \fH$ is contained in $W^Z$. Therefore, the image of $Q$ is contained in some irreducible component of $W^Z$, say $U$. \\

With this choice of $U,V,W$ we see that
\[\codim_{Y \times \fH^\ast} V = \codim_{\fH^\ast} Z\]
and
\[\codim_{Y \times \fH^\ast} W = \dim \fH^\ast,\]
whence
\[\codim_{Y \times \fH^\ast} W + \codim_{Y \times \fH^\ast} V = \dim \fH^\ast + \codim_{\fH^\ast} Z \geq \dim \fH^\ast + \dim Y.\]
Since 
\[\codim_{Y \times \fH^\ast} U \leq \dim \fH^\ast + \dim Y,\]
it follows that either $\dim U = 0$ or Bakker-Tsimerman applies, and the projection of $U$ to $Y$ is contained in a proper weak Mumford-Tate domain. The same applies to $Q$, which finishes the proof. \end{proof}

As in \cite{LV}, we need to transfer Lemma \ref{BT lem} to the $v$-adic setting. We note that the arguments in \cite{LV} only deal with the $\bQ_p$-case, but for our application the arguments are essentially the same. 

\begin{proposition} \label{period fib} Let $\Omega_v$ be the residue disk of $y_0$ in $Y(K_v)$. Then for any $K_v$-algebraic subvariety $Z \subset \fH_{K_v}^\ast$ satisfying (\ref{codim Z}), the set $\Phi_v^{-1}(Z)$ is not Zariski-dense in $Y$. Moreover, the number of irreducible components of $\Phi_v^{-1}(Z)$ depends only on $K, v$ and the $K_v$-isomorphism class of $Y$. \end{proposition}

\begin{proof}
Our strategy is to walk through the argument in \cite{LV} and keep track of what determines the bound on the number of components. The image $\mathfrak{R}=\phi(\Omega_v)$ is contained in a residue disk containing $\Phi_v(y_0)$ of the flag space $\fH_v$. In particular, there is some affine open set $\Spec A_v$ of $\fH_v$ containing $\Phi_v(y_0)$. We may suppose that $Z \subset \fH_v$, is cut out locally by equations $F_1,...,F_r$, where we suppose that $F_i \in A_v$, i.e., the $F_i$'s are regular functions on this affine open set. \\ 

Set $G_i=F_i\circ \phi$. These are formal power series, converging absolutely in $\Omega_v$. The common zero locus of these power series contains the fiber $\fF_y$ we are interested in. After choosing suitable coordinates we may work in the Tate algebra $K_v\langle x_1,...,x_N\rangle$, and consider

\[(G_1,...,G_r)\subseteq K_v\langle x_1,...,x_N\rangle=R \] Now choose $W=\Spec B$ in $Y$ containing the residue disk around $y_0$ with a natural map $B\ra R$. Our proof will then require the following claim: 

\begin{claim}
Let $\mathfrak{p}$ be a minimal prime ideals of $R$ that vanish at a point of $\Omega$, that is minimal among the primes containing $(G_1,...,G_r)$. Then there is a regular function $H$ in $\mathfrak{P}$.
\end{claim}

We will prove the claim as a separate lemma. Now let $\mathfrak{p}_1,\mathfrak{p}_2,...,\mathfrak{p}_t$ be the minimal prime ideals of $R$ that vanish at a point of $\Omega$. The number $t$ only depends on $G_{\Phi(y)}$ and the period mapping $\phi$. Thus we apply the claim $t$ times to obtain $H_1,..., H_t$ and from the construction we have that \[ \fF_y=\phi^{-1}(G_{\Phi(y)})\subseteq V(G_1,...,G_r) \subseteq V(H_1....H_t)\]
Since the number of the components of the $H_i$ are uniformly bounded, and the number of the $H_i$ depends on only on $G_{\Phi(y)}$ and $\phi$ the claim follows. \end{proof}

%\color{red}
%This is what I was worried about before, we choose some big affine open containing the image of the period map inside $G_{\Phi(y)}$. In the LV notion this is $Z$. Then we pull this back along the period mapping $\phi$. How do we know that we cannot choose BAD affine opens so that the number of prime ideas below blows up? The way I think this should work, is that $G_{\Phi(y)}$ is a linear algebraic group, so everywhere it should look the same because the group group action can be used to move points around. 
%\color{black}

%\end{proof}

\begin{lemma}
Let $\mathfrak{p}$ be a minimal prime ideals of $R$ that vanish at a point of $\Omega$, that is minimal among the primes containing $(G_1,...,G_r)$. Then there is a regular function $H$ in $\mathfrak{P}$.
\end{lemma}

\begin{proof} We may choose $B$ in such a way so that the number of prime ideals in $R$ is minimal. The ideal $\fp$ vanishes at some point of $\Omega_v$ by assumption, say at $y_0$. We then transfer the question to the complex numbers. Fix an isomorphism $\sigma: \ol{\bQ_p} \rightarrow \bC$, which induces an embedding $\sigma : K_v \rightarrow \bC$. Then $y_0$ gives rise to a point $y_0^\sigma \in Y(\bC)$, and the de Rham cohomology of $X_{y_0}^\sigma$ is obtained from that of $X_{y_0}$ via $\sigma$ by
\[H_{\dR}^\ast(X_{y_0}) \otimes_{K_v, \sigma)} \bC = H_{\dR}^\ast(X_{y_0^\sigma}/\bC).\]
We may regard the period map $\Phi_v$ as taking values in the Grassmanian $\fH_{K_v}$ for the left hand de Rham cohomology. \\

Let $U_\bC$ be a small complex neighbourhood of $y_0^\sigma$ and let $\Phi_{\bC} : U_\bC \rightarrow \fH_\bC$ be the complex period mapping, which we regard as taking values in the complex flag variety $\fH_\bC = \left(\fH_{K_v}\right)^\sigma$. Note we have the identification $\Phi_\bC(y_0^\sigma) = \Phi_v(y_0)^\sigma$.\\

$Z$ then gives rise to a complex algebraic variety $Z^\sigma \subset \fH_\bC$ and this subvariety satisfies
\begin{equation} \label{codim} \codim_{\fH} (Z) \geq \dim(Y).
\end{equation}
The functions $F_i$ are regular on $A_v$, an affine open, by hypothesis, whence $F_i^\sigma$ are defined over an affine open in $\fH_\bC$ containing $\Phi_\bC(y_0^\sigma)$, which locally cut out $Z^\sigma$. \\

Consider the completed local ring of $Y_v$ at $y_0$, which contains $G_i$. It is a formal power series ring over $K_v$, and $\sigma$ induces an injection from this completed local ring to the corresponding completed local ring of $Y_{\bC}$ at $y_0^\sigma$, which sends $G \mapsto G^\sigma$. Then we have
\begin{equation} \label{comp GI} G_i^\sigma = \text{ power series expansion of } F_i^\sigma \circ \Phi_\bC \text{ at } y_0^\sigma.
\end{equation}
This follows from the fact that the complex and $p$-adic period maps satisfy the same differential equation defined over $K$. It then follows from (\ref{comp GI}) that $G_i^\sigma$, a priori a complex formal power series, is in fact convergent in a small complex neighbourhood of $y_0^\sigma$; their vanishing locus for a sufficiently small such neighbourhood $V$ coincides with $\Phi_\bC^{-1}\left(Z^\sigma\right) \cap V$. \\

We now apply Lemma \ref{BT lem} to $Z^\sigma \subset \fH_{\bC}^\ast$, which shows that $\Phi_{\bC}^{-1}(Z^\sigma) \cap V \subset Y(\bC)$ is not Zariski dense in $Y(\bC)$. Moreover, the bounds from the argument of Bakker and Tsimerman uses functions in an $O$-minimal structure, and thus the finiteness produced is uniform. \\

We now analytically continue $\Phi_{\bC}$ from $V$ to a universal cover of $Y(\bC)$, obtaining finitely many irreducible components of $\Phi_{\bC}^{-1}(Z^\sigma)$ which intersect $V$. Applying Lemma \ref{BT lem} to each such component shows that the common zero-locus of $G_i^\sigma$ on $V$ is contained in the zero-locus of some algebraic function $G$. \\

Consider the ring $R_\bC = \bC \{x_1, \cdots, x_n\}$ of formal power series convergent in some neighbourhood of $0$. Given an ideal $I$ in this ring, we may associate a germ $V(I)$ of an analytic set at the origin. It then follows that the ideal of functions vanishing along this germ is precisely the radical $\sqrt{I}$ of $I$. Applying this to the ring of germs of holomorphic functions near $y_0^\sigma \in Y(\bC)$ and taking $I$ to be the ideal generated by the $G_i^\sigma$'s, we see that $\sqrt{I}$ is the ideal of functions vanishing on $V(I)$ and in particular contains $G$. It follows that there exists $m \geq 1$ such that $G^m \in I$.  \\

Hence the ideal spanned by $G_i^\sigma$ inside the ring of locally convergent power series contains the image of an algebraic function, that is, a regular function on some Zariski-open subset of $Y(\bC)$ containing $y_0^\sigma$. The same holds for formal power series, and pulling back this assertion via $\sigma^{-1}$ to $Y(\ol{\bQ_p})$. Thus, there exists a regular function $H$ in a neighbourhood of $y_0$ in $Y(\ol{\bQ_p})$ belonging to the ideal
\begin{equation} H \in \langle G_1, \cdots, G_n \rangle
\end{equation}
generated by the $G_i$'s in the completed local ring $\hat{\O}$ of $Y(\ol{\bQ_p})$ at $y_0$. Taking a norm, we may suppose that $H$ in fact arises from a regular function of $y_0$ in $Y(K_v)$. Multiplying by a suitable denominator if necessary, we may suppose without loss of generality that $H$ is regular on the chosen open affine around $y_0$, so that $H \in B_v$. Note that $B_v \otimes \ol{\bQ_p}$ surjects onto each quotient $\hat{\O}/\fm_{\hat{\O}}^t$, where $\fm_{\hat{\O}}$ is the maximal ideal of $\hat{\O}$. Therefore, for each $t \geq 1$, there are $Z_1, \cdots, Z_k \in B_v \otimes \ol{\bQ_p}$ such that
\begin{equation} \label{HZi} H \in \sum Z_i G_i + \fm_{\hat{\O}}^t.
\end{equation} 
By a simple linear algebra argument, we see that one can choose $Z_i \in B_v$ for all $i$. \\

The function $H$ then defines a rigid-analytic function on the residue disk $\Omega_v(y_0)$. Thus $H$ and $G_i$ both lie inside the Tate algebra $R$. Recall that we have fixed a prime ideal $\fp$ of $R$, contained in the maximal ideal $\fm$ associated to $y_0$, and containing the ideal $J = \langle G_1, \cdots, G_n \rangle \cap R$. Now  (\ref{HZi}) implies that 
\[H \in J + \fm^t\]
for each $t \geq 1$; thus the image of $H$ in $R/\fp$ lies in the intersection $\bigcap_{t \geq 1} \fm^t$. By Krull's intersection theorem applied to the Noetherian integral domain $R/\fp$, we have that the intersection of powers of $\fm$ is trivial. It therefore follows that $H \in \fp$, as desired.

%Then the claim is proven in \cite{LV}, furthermore the proof uses O-minimal geometry, this means that there is a uniform bound on the number of components of $V(H)$. The data that went into the above computation only depends on the affine local coordinates that we chosen for $\G_{\Phi(y)}$. Since $\G_{\Phi(y)}$ is a homogeneous space, this only depends on $\G_{\Phi(y)}$ itself. Furthermore, $\G_{\Phi(y)}$ is defined to be the quotient of $Z(\phi_\nu)$ by the stabilizer of a given filtration, that depends only on $y_0$. This depends only on $Z(\phi_\nu)$ and the point $y_0$ we fixed at the start. However in a $\nu$-adic residue disk every point is the centre, so the dependency on $y_0$ can be removed. \\
\end{proof}

What remains to be done is that we must show the number of possibilities of $Z(\phi_v)$ is bounded only in terms of $v$. Presumably this is really a statement about the classification of certain algebraic groups over $K_v$, but we have not been able to find such a result in the literature. We are thus compelled to give the following argument. 

\begin{lemma} \label{fin Z} The number of possibilities for $Z(\phi_v)$ depends only on $v$ and not on the residue disk $\Omega$, or even the curve $Y$. 
\end{lemma}

\begin{proof} Recall that $Z(\phi_v)$ is completely determined by $\phi_v$, which is then determined by the crystalline cohomology $H^q(X_y \times_{K} \ol{K}, \bQ_p)$. It is well-known that the crystalline cohomology and the Frobenius semi-linear operator $\phi_v$ depend only on the residue class modulo $v$; thus, $\phi_\nu$ depends only on the reduction of certain crystalline representations modulo $v$, which is then an abelian variety over a finite field $k_v$. It is well-known that the number of isomorphism classes of abelian varieties over $k_v$ is finite, and bounded only in terms of $|k_v|$; hence, there are at most a quantity, depending only on $v$, of possibilities for $\phi_v$, and thus the same number of possibilities for $Z(\phi_v)$. 
\end{proof}

Applying this to the curve case we obtain the following.

\begin{corollary}\label{MC}
Use the above notation. Suppose that $Y$ is a curve. Then $\fF_y$ has finitely many points, that depend only on $Z(\phi_v)$ and $\phi_{y_0}$.
\end{corollary}

In particular, we see that the intersections

\[p_3^{-1}(G_{\Phi(y)})\cap \fR \]

are finite, and the number of elements do not depend on $y$, only on the data used to construct $\fR$.

\section{The application to Mordell's conjecture: proof of Theorem \ref{MT}}

In \cite{LV} they construct a sequence of morphisms

\begin{equation} \label{KP fam} \xymatrix{X_q\ar[r]^\psi & H_q\ar[r]^\pi\ar[r] & Y} \end{equation}

called the \emph{Kodaira-Parshin} family, where $\pi$ is finite etale and $\psi$ is smooth of relative dimension $d_q=\dfrac{(q-1)(g-1)}{2}$ for some prime $q$. Furthermore, in the above set up we have that \[\overline{\textnormal{image}(\pi_1(Y(\bC),y_0))}\supset \prod_{y\in \pi^{-1}(y_0)}\textnormal{Sp}(H^1(X_y,\bQ),\omega)\]
We denote this condition as having full monodromy, it is intended to guarantee the dimension conditions that we will require in the end. To compare the result with what we need is as follows. \\

The construction of the Kodaira-Parshin family in \cite{LV} allows us to characterize the $K$-rational points of $Y$ as follows. If $E$ is a finite set on which $G_K = \Gal(\ol{\bQ}/K)$ acts, put $E(\Frob_v)$ for the subset of $E$ consisting of elements $e \in E$ for which the $\Frob_v$-orbit of $e$ has size at most $8$, and put 
\[\size_v(E) = \frac{|E(\Frob_v)|}{|E|}.\]
A feature of the family constructed by Lawrence and Venkatesh in \cite{LV}, which is critical for their proof of Mordell's conjecture, is that we have
\[Y(K) = Y(K)^\ast : = \left \{y \in Y(K) : \size_v(\pi^{-1}(y)) < \frac{1}{d+1} \right\}. \] 
Therefore, we can turn our attention to estimating the much more structured set $Y(K)^\ast$. Note that $Y(K)^\ast$ depends on the morphisms (\ref{KP fam}). \\

From the analysis given in Section 6 of \cite{LV}, it suffices to prove analogous statements to their Lemma 6.1 and 6.2. Lemma 6.2 of \cite{LV} is essentially the content of Proposition \ref{period fib}, which is our main contribution. Therefore, it remains to check that Lemma 6.1 gives a uniform bound. We restate their Lemma 6.1 for convenience:

\begin{lemma} \label{LV lem 6.1} There is a finite set $F \subset \Omega_v \cap Y(K)^\ast$ such that for all $y \in (\Omega_v \cap Y(K)^\ast) - F$, there exists $(y^\prime, w)$ with $y^\prime \in \pi^{-1}(y)$ and $w$ a place above $v$ such that $[K(y^\prime)_w : K_v] \geq 8$ and $\rho_{y^\prime}$ is simple as a $G_{K(y^\prime)}$-representation. 
\end{lemma}

What we need to check is that the size of $F$ is uniformly bounded in a twist family. We confirm this by noting that, following the proof of Lemma 6.1 in \cite{LV}, that the proof of the sublemma only depends on properties of the field $K$ and not on the curve $Y$, and that Lemmas 6.3 and 6.4 in \cite{LV} are purely statements about the linear algebra of $K_v$-vector spaces and hence independent of $Y$. Indeed, these three lemmas imply that the relevant set $\HH^{\mbox{bad}}$ is not Zariski dense in $\HH$, and thus we are in a position to apply Proposition \ref{not zar}. This completes the proof that our arguments, in addition to those given by Lawrence and Venkatesh in \cite{LV}, gives a uniform bound on the size of $Y(K)^\ast$ modulo the input coming from Faltings' lemma. \\

We now give a heuristic as to why such Galois representations should be uniformly bounded, provided that we insist that such representations have \emph{exact bad reduction} at $S$. In particular, we demand that such representations have bad reduction at every prime of $S$ and good reduction outside of $S$, rather than simply demanding that they have good reduction outside of $S$. Consider an elliptic curve $E/\bQ$ given by the usual Weierstrass model has square-free discriminant. It then follows that for every finite set $S$ of rational primes, there exists exactly one quadratic twist $E_d$ of $E$ such that $\Delta(E_d)$ is exactly divisible by the set of primes in $S$.

%%%%%%%%%%%%%%%%%%%%%%
\section{$S$-unit equations}

In this section we prove a uniform bound for the cardinality of the set
\begin{equation} \label{U1l} U_{L,t_0} = \{t \in \O_S^\ast :1 - t \in \O_S^\ast, t \not\in (K^\ast)^2, K(t^{1m}) \cong L, t \equiv t_0 \pmod{v}\}.
\end{equation}

We begin with the preliminary set-ups as in \cite{LV}. Let $S$ be a finite set of primes in $\O_K$, and let $\O_S$ be the ring of $S$-integers in $K$. Let $\O_S^\ast$ denote the ring of units in $\O_S$. Put
\begin{equation} U = \{t \in \O_S^\ast : 1 - t \in \O_S^\ast\}.
\end{equation} 
We may suppose that $S$ contains all primes in $K$ above $2$, and that $K$ contains the $8$-th roots of unity. Let $m$ be the largest power of $2$ dividing the order of the subgroup of roots of unity in $K$. By our assumption, we see that $m \geq 8$. \\

Put 
\[U_1 = \{t \in U : t \not \in (K^\ast)^2\}.\]
As was argued in \cite{LV}, it suffices to show that $U_1$ is a finite set. Since we are interested in uniform bounds, we note that this will only cost us by a factor of $m$. \\

Now suppose that $t \in U_1$. Since $t$ is a non-square by assumption and $\mu_m \subset K$, we see that the order of $t$ in the group $(K^\ast)/(K^\ast)^m$ is exactly $m$, for otherwise there is some proper divisor $k > 1$ of $m$ and an element $a \in K^\ast$ such that $t^k = a^m$, whence $t \in a^{m/k} \mu_k$, which shows that $t$ is a square. \\

Fixing $t^{1/m}$ an $m$-th root of $t$ in $\ol{\bQ}$, the field $K(t^{1/m})$ is Galois over $K$ since $K$ contains the $m$-th roots of unity. Moreover, elementary Galois theory shows that $\Gal(K(t^{1/m})/K) \cong \bZ/m\bZ$. \\

Let $\Y = \bP_{\O_S}^1 - \{0,1,\infty\}$, where $0, 1, \infty$ denote the corresponding sections over $\Spec \O_S$, and let $\Y^\prime = \bP_{\O_S}^1 - \{0, \mu_m, \infty\}$. Here $\mu_m$ denotes the set of $m$-th roots of unity. Let $\pi : \Y^\prime \rightarrow \Y$ be given by the map $u \mapsto u^m$. \\

Let $\X \rightarrow \Y^\prime$ be the \emph{Legendre family}, where the fibre $E_t$ over $t$ is given by
\[E_t : y^2 = x(x-1)(x-t).\]
Consider the composite 
\[\X \rightarrow \Y^\prime \overset{\pi}{\rightarrow} \Y;\]
this defines an abelian scheme over $\Y$. We then apply our results to the family $\X \rightarrow \Y$ and $X,Y$ the fibres of $\X, \Y$ respectively over $\Spec K$. The geometric fibre $X_t$ of $X \rightarrow Y$ over $t \in Y(K)$ is then given by the disjoint union of the curves $y^2 = x(x-1)(x - t^{1/m})$ over all $m$-th roots of $t$. \\

Our argument will rely on the following \emph{big monodromy} lemma of Lawrence and Venkatesh \cite{LV}:

\begin{lemma}[Big Monodromy] Consider the family of curves over $\bC - \{0,1\}$ whose fibre over $t \in \bC$ is the union of elliptic curves $E_z : y^2 = x(x-1)(x-z)$, over all $m$-th roots $z^m = t$. Then the action of monodromy
\begin{equation} \pi_1(\bC - \{0,1\}, t_0) \rightarrow \Aut \left(\bigoplus_{z^m = t_0} H_B^1 (E_z, \bQ) \right)
\end{equation}
has Zariski closure containing $\prod_z \SL (H_B^1 (E_z, \bQ))$. 
\end{lemma}

\subsection{Proof of Theorem \ref{SMT}} We diverge from the narrative in \cite{LV} by avoiding the use of Faltings' lemma. Indeed, we need to verify that the number of possible pairs
\begin{equation} \left( K_v \left(t^{1/m}\right), \rho_{t} | G_{K_v(t^{1/m})} \right) 
\end{equation}
is uniformly finite. \\

First we note that the since $K(t^{1/m}) \cong L$, that there are only finitely many possibilities for $K_v(t^{1/m}) \cong L_w$ for some prime $w$ in $L$ above $v$.  \\

For each such field $L_w$, we consider the $2$-dimensional crystalline representations $\rho_t | G_{L_w}$. We then invoke the following result of Brinon and Conrad, which is Theorem 8.3.6 in \cite{BC}:

\begin{lemma}[Brinon, Conrad] \label{BC lem} Let $\K$ be a finite extension of $\bQ_p$. The set of isomorphism classes of $2$-dimensional crystalline representations $V$ of $G_{\K}$, that have distinct Hodge-Tate weights $\{0, r\}$ with $r > 0$ and are not a direct sum of two characters is naturally parametrized by the set of quadratic polynomials $f(x) = x^2 + ax + b \in \O_{\K}[x]$ with $\ord_v(b) = r$, where $f$ is the characteristic polynomial associated to the crystalline representation. 
\end{lemma} 

\begin{remark} Brinon and Conrad actually proved the above lemma in the case $\K = \bQ_p$, but it is clear that their proof is purely geometric, whence applies for any finite extension of $\bQ_p$. 
\end{remark}

Lemma \ref{BC lem} then implies that in the $2$-dimensional case, the relevant crystalline representations are parametrized by the characteristic polynomials. Furthermore, since we are dealing with representations that come from restrictions of global Galois representations, such polynomials only depend on the relevant restrictions over $\bF_q$. Thus, by the same argument as in the proof of Lemma \ref{fin Z}, there are only finitely many possibilities for such representations, depending only on the number of $\bF_q$-equivalence classes of elliptic curves over $\bF_q$. With this observation, we then carry on with the proof following the strategy outlined in Section 4 of \cite{LV}. \\

Let $\rho_v$ be a fixed 2-dimensional $G_{K_v(t^{1/m})}$-representation. As before, we may use $p$-adic Hodge theory to associate a triple
\begin{equation} \left(H_{\dR}^{1} \left(X_{t, K_v}/K_v\right) \text{ as } K_v(t^{1/m})\text{-module}, \text{ Frobenius}, \text{ filtration} \right)
\end{equation}
where $H_{\dR}^1\left(X_{t, K_v}/K_v\right)$ is equipped with the structure of a 2-dimensional vector space over $K_v(t^{1/m})$. \\

Further, using the Gauss-Manin connection introduced in Section \ref{gen set}, we obtain an isomorphism 
\begin{equation} \label{SGM} H_{\dR}^1(X_{t, K_v}/K_v) \cong H_{\dR}^1 (X_{t_0, K_v}/K_v)
\end{equation}
which is compatible with the module structure over $H^0$'s. This induces an isomorphism $K_v(t^{1/m}) \cong K_v(t_0^{1/m})$. \\

The identification (\ref{SGM}) shows that the $F^1$-step of the filtration on $H_{\dR}^1(X_{t, K_v}/K_v)$ is identified with a $K_v(t_0^{1/m})$-line inside $H_{\dR}^1(X_{t_0, K_v}/K_v)$. This gives rise to a period map $\Phi$, defined by calling this line $\Phi(t)$. As before, we see that $\Phi(t)$ lies in a finite collection of orbits for 
\[Z = \text{ centralizer of } \Frob_v \text{ in } K_v(t_0^{1/m}) \text{-linear automorphisms of } H_{\dR}^1(X_{t_0, K_v}/K_v).\]
By Lemma 2.1 in \cite{LV}, we see that $\dim_{K_v} Z \leq 4$. We have thus shown that $\Omega_v$ is contained in $\Phi^{-1}(\Z)$, where $\Z \subset \Gr_{K_v}(2m, m)$ is a subvariety of dimension at most $4$. We can then apply Proposition \ref{not zar}, we have a uniform bounded on $|\Phi^{-1}(\Z)|$ provided that we verify the dimension hypothesis. \\

The dimension hypothesis is proved in the same manner as in \cite{LV}. Indeed, we note that their bounds for the size of monodromy depends only on the complex structure of those schemes involved, and thus does not depend on the set $S$ of primes. \\

It remains to deal with the possibility that the representation $\rho_v$ is a sum of two characters. This is handled by Lemma 4.4 in \cite{LV}, and we note that the argument only depends on the scheme $X$ over $Y = \bP^1 \setminus \{0, 1, \infty\}$ and the prime $v$. Therefore the finite count the Lemma produces does not depend on the set $S$, which is sufficient for our purposes.

\end{document}